\documentclass[11pt, leqno]{article}
\date{}
\usepackage{amssymb, amsbsy, amsmath, amsfonts, amssymb, amscd, mathrsfs, amsthm, dsfont}
\usepackage[english]{babel}
\usepackage[T1]{fontenc}
\usepackage{indentfirst}
\usepackage{color}

\makeatletter
\@addtoreset{equation}{section}
\makeatother

\newtheorem{statement}{}[section]
\newtheorem{theoreme}[statement]{Theorem}

\newcommand\D{\mathbb D}

\newcommand\e{{\rm e}}

\newcommand\eps{\varepsilon}
\newcommand\ind{{\rm 1\kern-.30em I}}

\let\phi=\varphi
\newcommand{\capa}{{\rm Cap}\,}
\newcommand{\capaN}{{\rm Cap}_N \,}

\title{\bf Entropy numbers and spectral radius type formula  for composition operators on the polydisk}
\author{\it Daniel Li,  Herv\'e Queff\'elec, Luis Rodr{\'\i}guez-Piazza }

\date{\footnotesize \today}

\begin{document}

\maketitle

\noindent {\bf Abstract.} We give estimates of the entropy numbers of composition operators on the Hardy space of the disk and of the polydisk. 
\medskip

\noindent {\bf MSC 2010} Primary 47B33 Secondary 47B06
\medskip

\noindent {\bf Key-words} composition operator ; entropy numbers ;  polydisk ; pluricapacity

\section{Introduction}

This short paper was motivated by a question of J.~Wengenroth  (\cite{WEN}) about entropy numbers of composition operators on Hardy spaces $H^2$, which 
stand a little apart in the jungle of ``$s$-numbers'', even though they seem the most natural for the study of compactness, since their membership in $c_0$ 
characterizes compactness, even in the general framework  of arbitrary Banach spaces. Indeed, in various papers (see 
\cite{BLQR, LIQUEROD, LIQURO, LQR, DHL}), 
we studied in detail the approximation numbers of composition operators, and here we will essentially transfer those results to entropy numbers thanks to the polar 
(Schmidt) decomposition and a general result on entropy numbers of diagonal operators on $\ell^2$.

So, the proofs are easy, but the statements feature a very different behavior of those entropy numbers. In particular, we will investigate a few properties related 
with a so-called ``spectral radius type formula'' which we  obtained, in dimension one through a result of Widom (\cite{LQR}), and, partially in 
dimension $N$ (\cite{DHL, LQR-pluricap}), through a result of Nivoche (\cite{NIV}) and Zakharyuta (\cite{ZAK}).
\bigskip

\noindent{\bf Acknowledgement.} L. Rodr{\'\i}guez-Piazza is partially supported by the project MTM2015-63699-P (Spanish MINECO and FEDER funds).
\goodbreak
\section{Entropy numbers}

We begin by recalling some facts on $s$-numbers.   

Given an operator $T \colon X \to Y$ between Banach spaces, recall (\cite{CAST}) that we can attach to this operator five non-increasing sequences 
$(a_n)$, $(b_n)$,  $(c_n)$, $(d_n)$, $(e_n)$ of non-negative numbers (depending on $T$), respectively the sequences of  \emph{approximation}, 
\emph{Bernstein}, \emph{Gelfand}, \emph{Kolmogorov}, and \emph{entropy} numbers of $T$. The latter are defined (\cite[Chapter 1]{CAST}, or 
\cite[Chapter 5]{PIS}), for $n\geq 1$ by:
\begin{equation} \label{entr} 
e_{n} (T) = \inf \{\varepsilon > 0 \, ; \ N \big( T (B_X), \varepsilon B_Y \big) \leq 2^{n - 1}\} \, , 
\end{equation}
where $B_X$ and $B_Y$ are the respective closed unit balls of $X$ and $Y$, and where, for $A, B \subseteq Y$, $N(A, B)$ denotes the smallest number of 
translates of $B$ needed to cover $A$.

All those sequences $(a_n)$, $(b_n)$,  $(c_n)$, $(d_n)$, $(e_n)$, say $(u_n)$, share the ideal property:
\begin{displaymath} 
u_{n} (A T B) \leq \Vert A \Vert \, u_{n}(T) \, \Vert B \Vert \, .
\end{displaymath} 

For Hilbert spaces, it turns out that $a_n = b_n = c_n = d_n = s_n$, where $(s_n)$  designates the sequence of singular numbers, but entropy numbers stay a little 
apart. 

For general Banach spaces $X$ and $Y$ and $T \colon X \to Y$, we have, in general (\cite[Theorem~1]{Carl}, see also \cite[Theorem~5.2]{PIS}), for 
$\alpha>0$:
\begin{displaymath} 
\sup_{1\leq k \leq n} k^\alpha e_k (T) \leq C_\alpha \sup_{1 \leq k \leq n} k^\alpha a_k (T) \, ,
\end{displaymath} 
and, if $X$ and $Y^\ast$ are of type $2$:
\begin{displaymath} 
\qquad \qquad a_n (T) \leq K \, e_n (T) \, , \qquad \text{for all } n \geq 1 
\end{displaymath} 
(\cite[Corollary~1.6]{GKS}), where $K = \kappa \, [ T_2 (X) T_2 (Y^\ast)]^2$; in particular, if $T$ acts between Hilbert spaces (see \cite[Theorem~5.3]{PIS}):
\begin{displaymath} 
\qquad \qquad a_n (T) \leq 4 \, e_n (T) \, , \qquad \text{for all } n \geq 1 \, .
\end{displaymath} 

Those inequalities indicate that entropy numbers are always bigger than singular numbers, up to a constant, and that, as far as the scale of powers $n^\alpha$ is 
implied, they are dominated by approximation numbers in a weak sense. But it turns out that, individually, they can be much bigger than the latter for composition 
operators, as we shall see. 

We will rely on the following  estimate  (\cite[p.~17]{CAST}), in which $\ell^2$ denotes the space of square-summable sequences $x = (x_k)_{k\geq 1}$ of 
complex numbers. This estimate is given  for the sequence $(\eps_{n})$ of covering numbers and with the scale of powers of $2$, but 
$e_n = \eps_{2^{n - 1}}$, by definition, and the change of $2$ to $\e$ only affects constants.
\begin{theoreme} {\rm (see \cite[p. 17]{CAST})} \label{irmt} 
There exist absolute constants $0 < a < b$ such that, for any diagonal compact operator $\Delta \colon \ell^2\to \ell^2$ with positive and non-increasing 
eigenvalues $(\sigma_k)_{k\geq 1}$, namely $\Delta \big( (x_k)_k \big) = (\sigma_k x_k)_k$, we have, for all $n \geq 1$:
\begin{equation} 
a \, \sup_{k \geq 1} \bigg[ \e^{- n /k} \bigg( \prod_{j = 1}^k \sigma_j \bigg)^{1/k} \bigg] 
\leq e_{n} (\Delta) \leq b \,  \sup_{k \geq 1} \bigg[ \e^{- n /k} \bigg( \prod_{j = 1}^k \sigma_j \bigg)^{1/k} \bigg] \, .
\end{equation} 
\end{theoreme}

A useful corollary of Theorem~\ref{irmt} is the following.
\begin{theoreme}\label{betis} 
Let $T \colon H_1 \to H_2$ be a compact operator between the Hilbert spaces $H_1$ and $H_2$, and let $(a_n)_{n \geq 1}$ be its sequence of approximation 
numbers. Then, for all $n \geq 1$:
\begin{equation} 
\alpha \, \sup_{k \geq 1} \bigg[ \e^{- n/ k} \bigg (\prod_{j = 1}^k a_j \bigg)^{1/k} \bigg] 
\leq e_{n} (T) \leq \beta \, \sup_{k\geq 1} \bigg[ \e^{- n /k} \bigg( \prod_{j = 1}^k a_j \bigg)^{1/k} \bigg] \, ,
\end{equation} 
where $\alpha$ and $\beta$ are positive numerical constants.
\end{theoreme}
\begin{proof} 
Let $T x = \sum_{n = 1}^\infty s_n (x \mid u_n) v_n$ the Schmidt decomposition of $T$, where $(u_n)_n$ and $(v_n)_n$ are orthonormal sequences of $H_1$ 
and $H_2$, respectively, and $(s_n)_n$ is the sequence of singular numbers of $T$. Let $\Delta \colon \ell_2 \to \ell_2$ the diagonal operator with diagonal values 
$s_n$, $n \geq 1$. Then $T = V_1 \Delta U_1$ and $\Delta = V_2 T U_2$, with $U_1 x = \big( (x \mid u_n)\big)_n$, 
$V_1 \big( (t_n)_n \big) = \sum_n t_n v_n$, $U_2 \big( (t_n)_n \big) = \sum_n t_n u_n$ and $V_2 x = \big( (x \mid v_n) \big)_n$. We have 
$\| U_1\|, \| V_1\|, \| U_2 \|, \| V_2\| \leq 1$; hence the result follows from Theorem~\ref{irmt} and the ideal property.
\end{proof}

This theorem might be thought useless, because we don't know better the $a_n$'s than the $e_n$'s! In our situation, this is not the case, since we made a more or 
less systematic study of the $a_n$'s for composition operators in \cite{BLQR, LIQURO, LIQUEROD, LQR} for example. 

We now pass to applications to composition operators $C_{\varphi}$, defined as $C_{\varphi} (f) = f \circ \varphi$ when they act on the Hardy space 
$H^{2} (\D^N)$ (which is always the case if $N = 1$). Here, $\varphi$ denotes an analytic and non-degenerate self-map of $\D^N$.  For clarity, we separate 
the cases of dimension $N = 1$ and of dimension $N \geq 2$. 
\goodbreak

\section{Applications in dimension $1$}

\subsection{General results}

In \cite{LQR}, we had coined the parameter:
\begin{equation} 
\beta_{1} (T) = \lim_{n \to \infty} \big[ a_{n}(T) \big]^{1/n} 
\end{equation} 
and its versions  $\beta_{1}^{+} (T), \beta_{1}^{-} (T)$ with a upper limit and a lower limit respectively. The following result (\cite{LQR}) shows in 
particular that no lower or upper limit is needed for $\beta = \beta_1$, and provides a simpler proof of the second item in Theorem~\ref{jfa} than in our initial 
proof of \cite{LIQUEROD}. 

For the definition of the Green capacity $\capa (A)$ of a Borel subset $A$ of $\D$, $0 \leq \capa (A) \leq \infty$, we refer to \cite{LQR}.

\begin{theoreme}\label{jfa} 
Let $\Omega = \varphi (\D)$, with $\phi \colon \D \to \D$ a non-constant analytic map. Then:\par\smallskip

$1)$ One always has $\beta_{1}^{-} (C_\varphi) = \beta_{1}^{+} (C_\varphi) =: \beta_{1} (C_\varphi)$ and:
\begin{equation} 
\beta_{1} (C_\varphi) = \exp [- 1/\capa (\Omega) ] > 0 \, .
\end{equation} 

$2)$ In particular, one has the equivalence:
\begin{equation} 
\beta_1 (C_\varphi) = 1 \quad \Longleftrightarrow \quad \Vert \varphi \Vert_\infty = 1 \, .
\end{equation} 
\end{theoreme}
\smallskip

Here, another parameter emerges.
\begin{equation}\label{crfo} 
\gamma_{1}(T) = \lim_{n\to \infty} \big[e_{n}(T) \big]^{1/\sqrt{n}}
\end{equation} 
and its  $\gamma_{1}^{+}(T), \gamma_{1}^{-}(T)$ versions. 

\begin{theoreme}\label{betise} 
Let   $\varphi \colon \D \to \D$ be a symbol and $\Omega = \phi (\D)$. Then:
\smallskip

$1)$ $\gamma_{1}^{-} (C_\varphi) = \gamma_{1}^{+} (C_\varphi) =: \gamma_{1} (C_\varphi)$ and:
\begin{equation} 
\gamma_{1} (C_\varphi) = \exp \big[- \sqrt{2/ \capa (\Omega)} \big] > 0 \, .
\end{equation} 

$2)$  In particular, one has the equivalence: 
\begin{equation} 
\gamma_{1}^{}(C_\varphi) = 1 \quad \Longleftrightarrow \quad \Vert \varphi\Vert_\infty = 1 \, .
\end{equation} 
\end{theoreme} 
\begin{proof} 
Set $\rho = 1/ \capa (\Omega)$ for simplicity of notations. Let $\varepsilon > 0$, and $C_\varepsilon$ a positive constant which depends only on $\varepsilon$ 
and can vary from a formula to another. Theorem~\ref{jfa} implies $a_k \leq C_\varepsilon \, \e^{\varepsilon k} \e^{-k\rho}$, whence:
\begin{displaymath} 
(a_1 \cdots a_k)^{1/k} \leq C_\varepsilon \, \e^{\varepsilon k/2} \e^{- \rho k/2} \, .
\end{displaymath} 
Theorem \ref{betis} now gives:
\begin{displaymath} 
e_{n} (C_\varphi) \leq C_\varepsilon \sup_{k\geq 1} \big[ \e^{\varepsilon k/2} \, \e^{- (n/k + \rho k/2}) \big] \, .
\end{displaymath} 
This supremum is essentially attained for $k = \big[\sqrt{2 n / \rho} \big]$ where $[\, . \, ]$ stands for the integer part, and gives:
\begin{displaymath} 
e_{n} (C_\varphi) \leq  C_\varepsilon  \e^{\varepsilon \sqrt{n / (2 \rho) }} \e^{- \sqrt{ 2 n \rho}} \, .
\end{displaymath} 
This implies $\gamma_{1}^{+} (C_\varphi) \leq \e^{\varepsilon \sqrt{1/(2\rho)}} \e^{- \sqrt{ 2\rho}}$, and finally:
\begin{displaymath} 
\gamma_{1}^{+} (C_\varphi) \leq \e^{- \sqrt{2\rho}} \, .
\end{displaymath} 

The lower bound $\gamma_{1}^{-} (C_\varphi) \geq \e^{- \sqrt{2\rho}}$ is proved similarly. 

This clearly ends the proof, since we know from \cite{LQR} 
that $\capa (\Omega) = \infty$ if ond only if $\Vert \varphi \Vert_\infty = 1$. 
\end{proof}
%

\subsection{Specific results}

For $0 < \theta < 1$, the lens map $\lambda_\theta$ of parameter $\theta$ is defined by:
\begin{equation} 
\lambda_\theta (z) = \frac{(1 + z)^\theta - (1 - z)^\theta}{(1 + z)^\theta + (1 - z)^\theta} 
\end{equation} 
(see \cite{SHA} or \cite{LIQUEROD}). 

\begin{theoreme} \label{lens} 
Let $\lambda_\theta$ be the lens map with parameter $\theta$.  Then, with positive constants $a$, $b$, $a'$, $b'$ depending only on $\theta$:
\begin{equation} 
a' \, \e^{- b' n^{1/3}} \leq e_{n} (C_{\lambda_\theta}) \leq a \, \e^{- b n^{1/3}} \, .
\end{equation} 
\end{theoreme}  
\begin{proof} 
We proved in \cite[Theorem~2.1]{LLQR} (see also \cite[Proposition~6.3]{LIQUEROD} that $a_k = a_k (C_{\lambda_\theta}) \leq a \, \e^{- b\sqrt{k}}$. It 
follows, using Theorem~\ref{betis}, that $(a_1 \cdots a_k)^{1/k}\leq a \, \e^{- b\sqrt{k}}$ and that, for some positive constant $C$:
\begin{displaymath} 
e_{n} (C_{\lambda_\theta}) \leq C \, \exp \big[- \big( (n/ k) + b k^{1/2} \big) \big] \, .
\end{displaymath} 
Taking $k = [n^{2/3}]$ gives the claimed upper bound. The lower bound is proved similarly, using the left inequality in Theorem~\ref{betis}, since we know 
(\cite{LQR}) that $a_k \geq a' \, \e^{- b' \sqrt{k}}$.
\end{proof}

We refer to \cite[Section~4.1]{LIQURO} for the definition of the cusp map $\chi$. We have: 
\begin{theoreme} \label{cusps} 
Let  $\chi$ be the cusp map. Then, with positive constants $a$, $b$, $a'$, $b'$:
\begin{equation} 
a' \, \e^{- b' \sqrt{n/ \log n}} \leq e_{n} (C_\chi) \leq a \, \e^{- b\sqrt{n/ \log n}} \, .
\end{equation} 
\end{theoreme} 
\begin{proof} 
We proved in \cite{LIQURO} that: 
\begin{equation} 
a' \, \e^{- b' k/ \log k}\leq a_k (C_\chi) \leq  a \, \e^{- b k/ \log k} \, .
\end{equation} 
The proof then follows the same lines as in Theorem~\ref{lens}, with the choice $k = [\sqrt{n\log n}]$.
\end{proof}
%
\section{The multidimensional case}

\subsection{General results}

Let $\varphi \colon \D^N\to \D^N$ be an analytic map. We will say that  $\varphi$ is non-degenerate if $\varphi (\D^N)$ has non-empty interior, equivalently 
if $\det \varphi \, ' (z) \neq 0$ for at least one point $z \in \D^N$. 
 
Let now $\varphi \colon \D^N\to \D^N$ be a non-degenerate analytic map inducing a bounded composition operator 
$C_\varphi \colon H^{2} (\D^N) \to H^{2}(\D^N)$ (this is not always the case as soon as $N > 1$, even if $\varphi$ is injective and hence non-degenerate, 
see for example \cite[p.~246]{COMA}, when the polydisk is replaced by the ball; but similar examples exist for the polydisk). Assume moreover that  $C_\phi$ is 
a compact operator.
\begin{theoreme} \label{venerdi} 
Let $C_\varphi \colon H^{2}(\D^N) \to H^{2}(\D^N)$ be a compact composition operator, with $\phi$ non-degenerate. We have:
\smallskip

$1)$ $e_{n}(C_\varphi)\geq c \, \exp \big(- C \, n^{\frac{1}{N+1}} \big)$, for some constants $C > c > 0$, depending on $\varphi$;
\smallskip

$2)$ if $\Vert \varphi\Vert_\infty < 1$, then  $e_{n}(C_\varphi)\leq C \, \exp \big(- c \, n^{\frac{1}{N+1}} \big)$, with $C > c > 0$ depending on $\varphi$.
\end{theoreme}
\begin{proof} 
$1)$ It is proved in \cite[Theorem~3.1]{BLQR} that, for a non-degenerate map $\varphi$, it holds:
\begin{displaymath} 
a_k (C_\phi) \geq a' \, \e^{- b' k^{1/N}} \, .
\end{displaymath} 
As in the previous section, it follows from Theorem~\ref{betis}, that $(a_1 \cdots a_k)^{1/k} \geq \e^{ - b'' k^{1/N}}$, and then, taking $k = [n^{N/(N+1)}]$, 
that:
\begin{displaymath} 
e_n (C_\phi) \geq c \, \e^{- C n^{1/(N+1)}} \, .
\end{displaymath} 

$2)$ Similarly, for $\| \phi \|_\infty < 1$, it is proved in \cite[Theorem~5.2]{BLQR} that:
\begin{displaymath} 
a_k (C_\phi) \leq C\, \e^{- c k^{1/N}} \, ;
\end{displaymath} 
and we get the result from Theorem~\ref{betis}.
\end{proof}
\smallskip

Those estimates motivate the introduction of the parameter: 
\begin{equation}\label{motiv} 
\gamma_{N} (C_\varphi) = \lim_{n\to \infty} \big[e_{n} (C_\varphi) \big]^{\frac{1}{n^{1/(N+1)}}} \, .
\end{equation}

We define similarly $\gamma_{N}^{\pm} (C_\varphi)$, and will say more on it in next section.

\subsection{Specific results}

\subsubsection{Multi-lens maps}

Let $\lambda_{\theta}$ be lens maps with parameter $\theta $. We define the multi-lens map $\Lambda_\theta$ of parameter $\theta$ on the polydisk $\D^N$ 
as:
\begin{equation} \label{multi-lens} 
\Lambda_\theta (z_1, \ldots, z_N) = \big( \lambda_{\theta} (z_1), \lambda_{\theta}(z_2), \ldots, \lambda_{\theta} (z_N) \big) \, , 
\end{equation} 
for $(z_1, \ldots, z_N) \in \D^N$. 

The following result is proved in \cite[Theorem~6.1]{BLQR}. 
\begin{theoreme} \label{lentille} 
Let $\Lambda_\theta$ be the multi-lens map with parameter $\theta$. Then, for positive constants $a$, $b$, $a'$, $b'$ depending only on $\theta$ and $N$, 
one has:
\begin{equation} \label{bi} 
a' \, \e^ {- b' n^{1/(2 N)}} \leq a_{n} (C_{\Lambda_\theta}) \leq a \, \e^ {- b \, n^{1/(2 N)}} \, .
\end{equation}
\end{theoreme}

The version of Theorem~\ref{lentille} for entropy numbers, stated without proof, is:
\begin{theoreme} \label{lenti} 
Let $\Lambda_\theta$ be the multi-lens map with parameter $\theta$. Then:
\begin{equation} \label{bis} 
a' \, \exp \, (- b' n^{1/(2 N + 1)}) \leq e_{n} (C_{\Lambda_\theta}) \leq a \, \exp \, (- b \, n^{1/(2 N + 1)}) \, .
\end{equation}
\end{theoreme}
%

\subsubsection{Multi-cusp maps}

Let $\chi \colon \D \to \D$ be the cusp map and $\phi \colon \D^N \to \D^N$ be the multi-cusp map defined by:
\begin{equation} 
\Xi \, (z_1, \ldots, z_N) = \big( \chi (z_1), \chi(z_2), \ldots, \chi (z_N) \big) \, . 
\end{equation} 
It is proved in \cite[Theorem~6.2]{BLQR}:
\begin{theoreme} \label{theo cusp}
Let $\chi \colon \D \to \D$ be the cusp map and $\Xi \colon \D^N \to \D^N$ be the multi-cusp map. Then:
\begin{equation} 
a' \, \e^{- b' \, n^{1/N} / \log n} \leq a_n (C_\Xi) \leq a \, \e^{- b \, n^{1/N} / \log n} \, ,
\end{equation} 
where $a$, $b$, $a'$, $b'$ are positive constants depending only on $N$.
\end{theoreme} 

The version of Theorem~\ref{theo cusp} for entropy numbers, stated without proof, is:
\begin{theoreme} \label{theo cus}
let $\chi \colon \D \to \D$ be the cusp map and $\Xi \colon \D^N \to \D^N$ be the multi-cusp map. Then:
\begin{equation} 
\begin{split}
a' \, \exp \big[- b' & \, n^{1 /(N+1)}\,  (\log n)^{- N/(N+1)}\big]  \\
& \qquad \leq e_n (C_\Xi) \leq a \, \exp\big[- b \, n^{1 /(N+1)}\,  (\log n)^{- N/(N+1)}\big] \, . 
\end{split}
\end{equation} 
\end{theoreme} 
%

  \section{Connections with pluricapacity and Zakharyuta's results}

Here, in dimension $N \geq 2$, the situation is satisfactory for upper bounds (see \cite{DHL}); for lower bounds, see \cite{LQR-pluricap}. 
The notion involved is now that of pluricapacity, or Monge-Amp\`ere capacity, coined by Bedford and Taylor in \cite{BETA}. More precisely, 
if $A$ is a Borel subset of $\D^N$, we refer to \cite{DHL} or \cite{LQR-pluricap} for the definition of its pluricapacity $\capaN(A)$, belonging to 
$[0,+\infty]$, and set:
\begin{align}
\tau_{N}(A) & = \frac{1}{(2\pi)^N} \, \capaN (A) \label{tau} \\ 
\Gamma_{N}(A) & = \exp \bigg[- \Big(\frac{N!}{\tau_{N}(A)}\Big)^{1/N} \bigg] \label{gamma}  \\
\beta_{N}^{+}(T) & = \limsup_{n \to \infty}\big[a_{n}(T)\big]^{1/ n^{1/N}} \, . \label{tabe} 
\end{align}

We temporarily assume that $\Vert \varphi\Vert_\infty < 1$ so that $K = \overline{\varphi (\D^N})$ is a compact subset of $\D^N$.  We proved in 
\cite[Theorem~6.4]{DHL}, relying on positive results of Nivoche (\cite{NIV}) and Zaharyuta (\cite[Proposition~6.1]{ZAK}) on the so-called Kolmogorov 
conjecture, that:
\begin{theoreme}\label{citation} 
It holds:
\begin{equation} 
\beta_{N}^{+} (C_\varphi) \leq \Gamma_{N}(K) \, .
\end{equation} 
\end{theoreme}

We have the following result, which extends the previous result in dimension $1$. \goodbreak
\begin{theoreme}\label{nose}
The following upper bound holds:
\begin{equation} 
\gamma_{N}^{+} (C_\varphi)\leq \exp\big(-\beta_N\rho^{N / (N + 1)}\big) \, ,
\end{equation} 
where:
\begin{equation} \label{nese}  
\quad \rho = \bigg( \frac{N!}{\tau_N (K)} \bigg)^{1/N} = 2 \pi \bigg( \frac{N!}{\capaN (K)}\bigg)^{1/N} \, , 
\end{equation}
and
\begin{equation} \label{ugli}
\begin{split} 
\beta_N = \bigg(\frac{N}{N + 1} \bigg)^{N / (N + 1)} \big( N^{- N / (N + 1)} + & N^{1 / (N + 1)} \big) \\
& \geq \e^{- 1 / (N + 1)} N^{1 / (N + 1)}  \, .
\end{split}
\end{equation} 
\end{theoreme}
\begin{proof} 
Abbreviate $a_{n} (C_\varphi)$ and $e_{n} (C_\varphi)$ to $a_{n}$ and $e_{n}$, and set $\alpha = N / (N + 1)$. Let $\varepsilon > 0$. 
Theorem~\ref{citation} implies:
\begin{displaymath} 
a_{k}\leq C_{\varepsilon} \, \e^{\varepsilon k^{1/N}} \e^{-\rho k^{1/N}} \, ,
\end{displaymath} 
so:
\begin{displaymath} 
(a_1 \cdots a_{k})^{1/k} \leq C_{\varepsilon} \, \e^{\varepsilon k^{1/N}} \e^{- \rho \alpha k^{1/N}} \, .
\end{displaymath} 
Apply once more Theorem~\ref{betis} to obtain:
\begin{displaymath} 
e_{n} \leq C_{\varepsilon} \, \sup_{k\geq 1} \e^{\varepsilon k^{1/N}} \exp \big[- (n/k + \rho \alpha k^{1/N}) \big] \, .
\end{displaymath} 
The supremum is essentially attained for $k$ the integral part of $(N / \rho \alpha)^{\alpha} n^{\alpha}$ and then, in view of \eqref{ugli} and 
$\alpha / N = 1 - \alpha$, up to a negligible term: 
\begin{align*}
\frac{n}{k} + \rho \, \alpha \, k^{1/N}
& = n^{1 - \alpha}\bigg( \frac{\rho \, \alpha}{N} \bigg)^{\alpha} + \rho \, \alpha \, n^{1 - \alpha} \bigg( \frac{N}{\rho\alpha}\bigg)^{1 - \alpha} \\
& = n^{1 - \alpha} (\rho \, \alpha)^{\alpha} (N^{- \alpha} + N^{1 - \alpha}) \, .
\end{align*}
Finally, 
\begin{displaymath} 
e_{n} \leq C_{\varepsilon} \, \e^{\varepsilon n^{1 - \alpha}} \exp \, (- \beta_{N} \rho^{\alpha} n^{1 - \alpha}) 
= C_{\varepsilon} \, \e^{\varepsilon n^{1/(N+1)}} \exp \, (- \beta_{N} \rho^{\alpha} n^{1/(N+1)}) \, .
\end{displaymath} 

This clearly ends the proof of Theorem~\ref{nose}.
\end{proof}
\medskip

\noindent {\bf Remark.}
We have so far no sharp lower bound for entropy numbers, at least when $\Vert \varphi\Vert_\infty = 1$, since we already fail to have one in general for 
approximation numbers (see however \cite{LQR-pluricap}). 

Besides, let $J \colon H^{\infty} (\D^N) \to \mathcal{C}(K)$ be the canonical embedding, when $K\subseteq \D^N$ is a ``condenser'',  namely a compact 
subset of $\D^N$ such that any bounded analytic function on $\D^N$ which vanishes on $K$ vanishes identically, which is moreover ``regular''. The positive 
solution to the Kolmogorov conjecture can be expressed in terms of the Kolmogorov numbers $d_{n} (J)$ of $J$ or equivalently, in terms of the entropy numbers 
$e_{n} (J)$ of $J$ (\cite[Theorem~5]{ZAK-1}, generalizing Erokhin's result in dimension $1$ appearing in his posthumous paper \cite{Erokhin} and methods 
due to Mityagin \cite{Mityagin} and Levin and Tikhomirov \cite{LT}; see also \cite[Lemma~2.2]{ZAK}). The result is that, taking 
$K=\overline{\varphi (\D^N})$, one has, with sharp constants $c_K$, $c'_K$ depending on the pluricapacity of $K$ in $\D^N$:
\begin{equation} 
d_{n} (J) \approx \e^{- c_K \, n^{1/N}} \quad \text{and} \quad e_{n} (J) \approx \e^{- c'_K\, n^{1/(N+1)}} \, .
\end{equation} 
This jump from the exponent $1/N$ to the exponent $1/(N+1)$ is reflected in our Theorem~\ref{nose}, through the new parameter $\gamma_{N}^{+}$.


\smallskip

{\footnotesize
Daniel Li \\ 
Univ. Artois, Laboratoire de Math\'ematiques de Lens (LML) EA~2462, \& F\'ed\'eration CNRS Nord-Pas-de-Calais FR~2956, 
Facult\'e Jean Perrin, Rue Jean Souvraz, S.P.\kern 1mm 18 
F-62\kern 1mm 300 LENS, FRANCE \\
daniel.li@euler.univ-artois.fr
\smallskip

Herv\'e Queff\'elec \\
Univ. Lille Nord de France, USTL,  
Laboratoire Paul Painlev\'e U.M.R. CNRS 8524 \& F\'ed\'eration CNRS Nord-Pas-de-Calais FR~2956 
F-59\kern 1mm 655 VILLENEUVE D'ASCQ Cedex, FRANCE \\
Herve.Queffelec@univ-lille.fr
\smallskip
 
Luis Rodr{\'\i}guez-Piazza \\
Universidad de Sevilla, Facultad de Matem\'aticas, Departamento de An\'alisis Matem\'atico \& IMUS,  
Calle Tarfia s/n \\ 
41\kern 1mm 012 SEVILLA, SPAIN \\
piazza@us.es
}

\end{document}